\documentclass[a4paper]{amsart}
\usepackage[margin=1in]{geometry}
\usepackage{verbatim}
\usepackage{multicol}

\usepackage{tikz}
\usetikzlibrary{calc}
\usetikzlibrary{arrows}
\usetikzlibrary{decorations.markings}
\usetikzlibrary{decorations.pathreplacing}
\usetikzlibrary{arrows,shapes,positioning}
\usetikzlibrary{patterns}
\tikzstyle directed=[postaction={decorate,decoration={markings,
    mark=at position #1 with {\arrow{>}}}}]
\tikzstyle rdirected=[postaction={decorate,decoration={markings,
    mark=at position #1 with {\arrow{<}}}}]
\tikzset{anchorbase/.style={baseline={([yshift=-0.5ex]current bounding box.center)}}}

\tikzset{
    partial ellipse/.style args={#1:#2:#3}{
        insert path={+ (#1:#3) arc (#1:#2:#3)}
    }
  }

\usepackage{tikz-cd}

\usepackage{todonotes}

\usepackage{hyperref}
\usepackage[capitalize]{cleveref}
\crefname{figure}{Figure}{Figures}

\newcommand{\C}{\mathbb{C}}
\newcommand{\Z}{\mathbb{Z}}

\newcommand{\B}{{\bf B}}

\newcommand{\calC}{\mathcal{C}}

\newcommand{\id}{\mathrm{id}}

\DeclareMathOperator{\Cone}{Cone}
\DeclareMathOperator{\Hom}{Hom}
\DeclareMathOperator{\End}{End}

\newcommand{\HOM}{\ensuremath{\Hom_{\mathcal{C}_\Gamma}}}
\newcommand{\END}{\ensuremath{\End_{\mathcal{C}_{\Gamma}}}}

\renewcommand{\tilde}{\widetilde}

\theoremstyle{theorem}
\newtheorem{theorem}{Theorem}[section]
\newtheorem{lemma}[theorem]{Lemma}
\newtheorem{corollary}[theorem]{Corollary}
\newtheorem{remark}[theorem]{Remark}
\newtheorem{definition}[theorem]{Definition}
\newtheorem{example}[theorem]{Example}

\title[Burau representation of Artin--Tits groups]{Some remarks about the faithfulness of the Burau representation of Artin--Tits groups}

\author{Asilata Bapat}
\address{Mathematical Sciences Institute \\ Australian National University}
\email{asilata.bapat@anu.edu.au}

\author{Hoel Queffelec}
\address{IMAG\\ Univ. Montpellier\\ CNRS \\ Montpellier \\ France \\ {\it and} Mathematical Sciences Institute \\ Australian National University}
\email{hoel.queffelec@umontpellier.fr}

\begin{document}

\maketitle

\begin{abstract}
  We discuss the extension of the faithfulness question for the Burau representation of braid groups to the case of Artin--Tits groups. We prove that the Burau representation is not faithful in affine type $\tilde{A_3}$, and not faithful over several finite rings in type $D_4$, using an algorithmic approach based on categorical methods that generalize Bigelow's curve strategy outside of type $A$.
\end{abstract}

\section{Introduction}

The Burau representation, that was defined 90 years ago~\cite{Burau}, has for long been the major tool in the study of braids, in relation to the Alexander polynomials of knots. Despite meritorious efforts, its faithfulness is not completely known: it is a classical result (see \cite{MagnusPeluso} and the discussion in~\cite{Birman_book}) that it is faithful in type $A_2$ (3-strand braid group), but the fact that it is not faithful in type $A_4$ and above was only discovered step by step in the 90's~\cite{Moody,LongPaton,Bigelow_Burau}. And type $A_3$ still resists all efforts~\cite{CooperLong,BeridzeTraczyk,Datta,GibsonWilliamsonYacobi}, including computer-based ones~\cite{FullartonShadrach} (see also the discussion in~\cite{BharathramBirman}).

In this paper we build on the observation that the definition of the Burau representation makes perfect sense outside of type $A$. In this larger generality, faithfulness is open in a fairly large class of cases. We investigated the first cases using computational methods, generalizing Bigelow's curve-based approach outside of type $A$, where Khovanov--Seidel's category~\cite{KhS} replaces curves when we lack geometric foundations. It should be highlighted that despite its theoretical definition, this categorical representation is completely computable.

In~\cref{thm:unfaithful-affine-a3}, we exhibit a counterexample to faithfulness in affine type \(A_3\).
In~\cref{thm:d4-unfaithful-finite-rings}, we exhibit a range of counterexamples over several finite rings in type \(D_4\).
These counterexamples were found by running Python and Sage searches on the Australian National Computational Infrastructure (NCI Australia)~\cite{nci}.
The counterexample in affine type \(A_3\) was found using a variation of Bigelow's strategy from~\cite{Bigelow_Burau}.
The counterexamples in type \(D_4\) were found by mixing variants of Bigelow's strategy with the one developed by Gibson, Williamson and Yacobi in~\cite{GibsonWilliamsonYacobi}.
Our code is available on a Git repository~\cite{BQ_git}.
To our surprise, all other cases resisted our investigations over $\Z$.

\subsection*{Acknowledgements}
We would like to warmly thank Anand Deopurkar, Anthony Licata and Thomas Haettel for countless valuable discussions, and we are grateful to Oded Yacobi for sharing insights about the strategy used in~\cite{GibsonWilliamsonYacobi}. Many thanks as well to Pascal Azerad and Cl\'ement Maria for coding tips.
H.Q. has been funded by the European Union's Horizon 2020 research and innovation programme under the Marie Sklodowska-Curie grant agreement No 101064705.
A.B. was partly supported by the Australian Research Council grants DE240100447 and DP240101084.
This research was undertaken with the use of the National Computational Infrastructure (NCI Australia).
NCI Australia is enabled by the National Collaborative Research Infrastructure Strategy (NCRIS).

\section{The (generalised) Burau representation}
In this section we recall the definition of the Burau representation for the Artin--Tits group associated to any Coxeter diagram. Recall that such a Coxeter diagram $\Gamma$ is a loop-free finite graph with vertex set $V(\Gamma)$ and unoriented edges.
Further, each edge \(i-j\) is labelled by $m_{i,j}$ in $\{3,\ldots,\infty\}$.
Then the associated Artin--Tits group $\B(\Gamma)$ is the group presented as:
\[
\B(\Gamma):=\langle \sigma_i\text{ for }i\in V(\Gamma) \mid \underbrace{\sigma_i\sigma_j\cdots}_{m_{ij}}=\underbrace{\sigma_j\sigma_i\cdots}_{m_{ij}}\rangle.
\]
By convention, $m_{ij}=\infty$ means that there is no relation involving the letters $\sigma_i$ and $\sigma_j$.
If $i$ and $j$ are not adjacent, we set $m_{ij}=2$.

The associated Coxeter group has the same presentation, with the added relation that all generators are of order $2$:
\[
G_\Gamma:=\B(\Gamma)/\langle \sigma_i^2=1\rangle.
\]

We recall the \emph{geometric representation} of the Coxeter group \(G_{\Gamma}\) (see, e.g.,~\cite[Sect.~4.2]{BjornerBrenti}), from which one can study the group properties.
The underlying complex vector space \(V_{1,\Gamma}\) is freely generated by the roots \(\{\alpha_i \mid i \in V(\Gamma)\}\):
\[
V_{1, \Gamma}=\mathbb{C}\langle \alpha_i\rangle_{i \in V(\Gamma)},
\]
and is equipped with a non-degenerate pairing as follows:
\[
\langle \alpha_i,\alpha_j\rangle =
\begin{cases}
  2 & \text{if } i=j; \\
  0 & \text{if } m_{ij}=2\;\text{($i$ and $j$ not adjacent)};\\
  -2\cos(\frac{\pi}{m_{ij}})& \text{if}\;m_{ij}\neq \infty; \\
  -2& \text{if } m_{ij}=\infty.
\end{cases}
\]

Note that the $m_{ij}=2$ case is compatible with the definition for general $m_{ij}$.
Now $G_\Gamma$ acts on $V_{1, \Gamma}$ by reflections via the following formula:
\[
s_i(\alpha_j)=\alpha_j-\langle \alpha_i,\alpha_j\rangle \alpha_i.
\]
A key point is that this representation is faithful (see for example~\cite[Thm 4.2.7]{BjornerBrenti}).

Passing to Artin--Tits groups, this geometric representation $q$-deforms into the Burau representation.
Let us first consider \(V_{q, \Gamma}\), which is the \(\mathbb{C}(q)\)-vector space \(\mathbb{C}(q) \otimes_{\mathbb{C}}V_{1, \Gamma}\).
Then \(V_{q, \Gamma}\) has a \(q\)-deformed pairing, defined as follows.
On the basis elements, we set:
\[
\langle \alpha_i,\alpha_j\rangle =\begin{cases}
1+q^2 \;\text{if}\;i=j \\
2q\cos(\frac{\pi}{m_{ij}})\; \text{if}\;i\neq j\;\text{and}\; m_{ij}<\infty \\
2q\;\text{if}\;m_{ij}=\infty.
\end{cases}
\]
The pairing is extended to \(V_{q,\Gamma}\) by \(\mathbb{C}\)-bilinearity, and sesquilinearity with respect to \(q\) as follows:
\[\langle qx,y \rangle = q^{-1}\langle x,y \rangle \text{ and } \langle x,qy \rangle = q\langle x,y \rangle\]
Taking $q=-1$ takes us back to the undeformed setting.

\begin{definition}\label{def:burau-rep}
  The \emph{Burau representation} is an action of $\B(\Gamma)$ on $V_{q, \Gamma}$ defined on generators $\sigma_i\in \B(\Gamma)$ by:
  \[
    \sigma_i(\alpha_j)=\alpha_j-\langle \alpha_i,\alpha_j\rangle \alpha_i.
  \]
\end{definition}
One can check the following formula for the inverses of the generators:
\[
\sigma_i^{-1}(\alpha_j)=\alpha_j-q^{-2}\langle \alpha_j,\alpha_i\rangle \alpha_i.
\]

\begin{example}
  In simply-laced type, one can check that:
  \[
    \sigma_i(\alpha_j)=\begin{cases} -q^2\alpha_i \;\text{if}\; i=j, \\
      \alpha_i-q\alpha_j \;\text{if}\; m_{i,j}=3, \\
      \alpha_j\;\text{otherwise}.
      \end{cases}
    \]
\end{example}

\section{Summary of results and approaches}

\subsection{Faithfulness questions for generalized Burau representations}

Studying the faithfulness of the Burau representation in type $A$ has for long been a tantalizing question.
It is known to be faithful in type \(A_2\), which corresponds to the 3-strand braid group (see~\cite{MagnusPeluso} or~\cite{Birman_book} and references therein), and unfaithful for at least 5 strands~\cite{Moody, LongPaton,Bigelow_Burau}.
The $4$-strand case (corresponding to type \(A_3\)) is open.

In light of the generalization of the Burau representation to larger types, it is a natural question to ask about faithfulness in other types. With type $A_3$ in mind, being able to identify a kernel in nearby types (like $D_4$, $\tilde{A_3}$, \dots) might help identifying properties of elements in the kernel that would yield proof strategies for faithfulness.

In~\cref{thm:unfaithful-affine-a3}, we show that the Burau representation is not faithful over \(\mathbb{Z}\) in affine type $\tilde{A_3}$.
In~\cref{thm:d4-unfaithful-finite-rings} we show that the Burau representation is not faithful in type $D_4$ over finite rings $\Z/p\Z$ with $p\leq 16$.

Towards approaching more general types, the following observation is clear from the fact that Artin--Tits groups include into each other in a way that is compatible with the Burau representation.
Recall that a subgraph of an undirected graph is called a \emph{full subgraph} if its vertex set is a subset of the vertices, and its edge set contains all of the edges in the original graph between any pair of vertices in the subset.
\begin{lemma}
  If a Coxeter graph $\Gamma$ contains $A_4$ or \(\widetilde{A_3}\) as a full subgraph, then the Burau representation of $B_\Gamma$ is not faithful.
\end{lemma}
\begin{proof}
  We know from existing results that the Burau representation is not faithful in type \(A_n\) for \(n = 4\) and beyond, and we know from~\cref{thm:unfaithful-affine-a3} that it is not faithful in type \(\widetilde{A_3}\).
  Thus for any graph that contains these as full subgraphs, the kernel of the Burau representation will be non-trivial.
\end{proof}
Combined with the fact that faithfulness is not known in type \(A_3\), this lemma leaves us with a large range of open cases.
Namely, all Coxeter graphs \(\Gamma\) that do not contain \(A_4\) or \(\widetilde{A_3}\) as full subgraphs.
Among these, we investigated the following ones (we restricted to the simply-laced case).
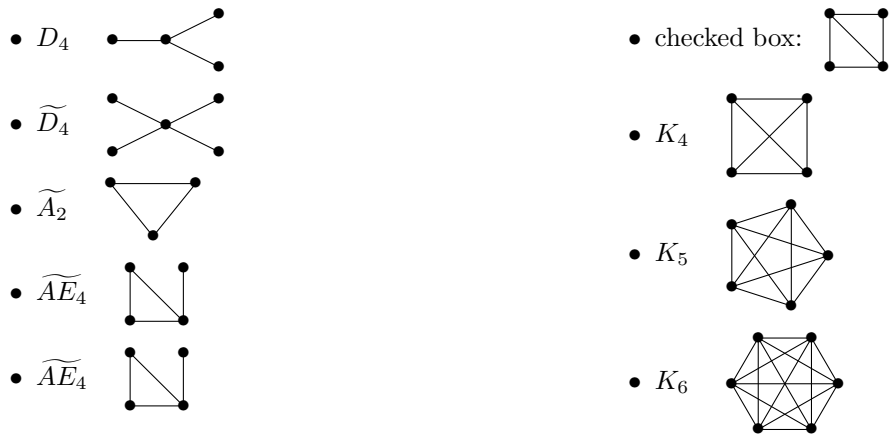
\begin{multicols}{2}
\begin{itemize}
\item $D_4 \quad
  \begin{tikzpicture}[anchorbase, scale=.7] 
    \node (1) at (1,0) {\textbullet};
    \node (2) at (2,0) {\textbullet};
    \node (3) at (3,.5) {\textbullet};
    \node (4) at (3,-.5) {\textbullet};
    \draw (1.center) -- (2.center);
    \draw (2.center) -- (3.center);
    \draw (2.center) -- (4.center);
  \end{tikzpicture}$
\item $\tilde{D_4} \quad
  \begin{tikzpicture}[anchorbase, scale=.7] 
    \node (0) at (1,.5) {\textbullet};
    \node (1) at (1,-.5) {\textbullet};
    \node (2) at (2,0) {\textbullet};
    \node (3) at (3,.5) {\textbullet};
    \node (4) at (3,-.5) {\textbullet};
    \draw (0.center) -- (2.center);
    \draw (1.center) -- (2.center);
    \draw (2.center) -- (3.center);
    \draw (2.center) -- (4.center);
  \end{tikzpicture}$
\item $\tilde{A_2}  \quad \begin{tikzpicture}[anchorbase, scale=.7] 
    \node (0) at (.8,-1) {\textbullet};
    \node (1) at (0,0) {\textbullet};
    \node (2) at (1.6,0) {\textbullet};
    \draw (0.center) -- (1.center);
    \draw (2.center) -- (0.center);
    \draw (1.center) -- (2.center);
  \end{tikzpicture}$
\item $\tilde{AE_4}  \quad \begin{tikzpicture}[anchorbase, scale=.7] 
    \node (1) at (0,0) {\textbullet};
    \node (2) at (0,-1) {\textbullet};
    \node (3) at (1,-1) {\textbullet};
    \node (4) at (1,0) {\textbullet};
    \draw (1.center) -- (2.center);
    \draw (1.center) -- (3.center);
    \draw (2.center) -- (3.center);
    \draw (3.center) -- (4.center);
  \end{tikzpicture}$
\item $\tilde{AE_4}  \quad \begin{tikzpicture}[anchorbase, scale=.7] 
    \node (1) at (0,0) {\textbullet};
    \node (2) at (0,-1) {\textbullet};
    \node (3) at (1,-1) {\textbullet};
    \node (4) at (1,0) {\textbullet};
    \draw (1.center) -- (2.center);
    \draw (1.center) -- (3.center);
    \draw (2.center) -- (3.center);
    \draw (3.center) -- (4.center);
  \end{tikzpicture}$
\item checked box: $ \begin{tikzpicture}[anchorbase, scale=.7] 
    \node (1) at (0,0) {\textbullet};
    \node (2) at (0,-1) {\textbullet};
    \node (3) at (1,-1) {\textbullet};
    \node (4) at (1,0) {\textbullet};
    \draw (1.center) -- (2.center);
    \draw (1.center) -- (3.center);
    \draw (1.center) -- (4.center);
    \draw (2.center) -- (3.center);
    \draw (3.center) -- (4.center);
  \end{tikzpicture}$
\item $K_4\quad  \begin{tikzpicture}[anchorbase, scale=.7] 
    \node (1) at (45:1) {\textbullet};
    \node (2) at (135:1) {\textbullet};
    \node (3) at (-135:1) {\textbullet};
    \node (4) at (-45:1) {\textbullet};
    \draw (1.center) -- (2.center);
    \draw (1.center) -- (3.center);
    \draw (1.center) -- (4.center);
    \draw (2.center) -- (3.center);
    \draw (2.center) -- (4.center);
    \draw (3.center) -- (4.center);
\end{tikzpicture}$
\item $K_5\quad  \begin{tikzpicture}[anchorbase, scale=.7] 
    \node (1) at (0:1) {\textbullet};
    \node (2) at (72:1) {\textbullet};
    \node (3) at (144:1) {\textbullet};
    \node (4) at (216:1) {\textbullet};
    \node (5) at (288:1) {\textbullet};
    \draw (1.center) -- (2.center);
    \draw (1.center) -- (3.center);
    \draw (1.center) -- (4.center);
    \draw (1.center) -- (5.center);
    \draw (2.center) -- (3.center);
    \draw (2.center) -- (4.center);
    \draw (2.center) -- (5.center);
    \draw (3.center) -- (4.center);
    \draw (3.center) -- (5.center);
    \draw (4.center) -- (5.center);
\end{tikzpicture}$
\item $K_6\quad  \begin{tikzpicture}[anchorbase, scale=.7] 
    \node (1) at (0:1) {\textbullet};
    \node (2) at (60:1) {\textbullet};
    \node (3) at (120:1) {\textbullet};
    \node (4) at (180:1) {\textbullet};
    \node (5) at (240:1) {\textbullet};
    \node (6) at (300:1) {\textbullet};
    \draw (1.center) -- (2.center);
    \draw (1.center) -- (3.center);
    \draw (1.center) -- (4.center);
    \draw (1.center) -- (5.center);
    \draw (1.center) -- (6.center);
    \draw (2.center) -- (3.center);
    \draw (2.center) -- (4.center);
    \draw (2.center) -- (5.center);
    \draw (2.center) -- (6.center);
    \draw (3.center) -- (4.center);
    \draw (3.center) -- (5.center);
    \draw (3.center) -- (6.center);
    \draw (4.center) -- (5.center);
    \draw (4.center) -- (6.center);
    \draw (5.center) -- (6.center);
\end{tikzpicture}$
\end{itemize}
\end{multicols}
A number of other graphs are possible, including non-simply-laced ones, as well as more complicated simply-laced ones such as many-pointed stars (continuing from \(D_4\) and \(\widetilde{D}_4\)), and complete graphs \(K_n\) of higher order.
As far as we are aware, all of these cases (and other similar ones) remain open over \(\mathbb{Z}\).
It was surprising to us that none of the other possibilities gave us any answers.

\subsection{Proof strategies (after Bigelow)}\label{subsec:proof-strategies}
For both~\cref{thm:unfaithful-affine-a3,thm:d4-unfaithful-finite-rings}, we used adaptations of Bigelow's proof strategy from~\cite{Bigelow_Burau} for the lack of faithfulness in type $A_4$ and above.
We recall this strategy now, together with our modified versions.
The proof is based on the fact that the braid group $\B_n=\B_{A_{n-1}}$ is the mapping class group of the disk with $n$ points removed, which acts transitively on simple curves connecting two of the points.
See~\cref{fig:actiononcurves} for an illustration; note that the numbering of the points is fixed, and we do not carry it along when we act on curves. The centralizer of the full set of curves is the center of the braid group, generated by the full twist.
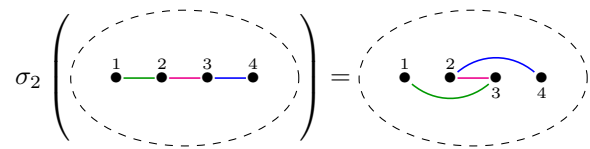
\begin{figure}[!h]
  \[
  \sigma_2\left(
    \begin{tikzpicture}[anchorbase,scale=.6]
    \tikzstyle{point}=[inner sep=0em]      
    \draw[dashed] (0,0) ellipse (2.5cm and 1.5cm);
    \node[point] (1) at (-1.5,0) {$\bullet$};
    \node[point] (2) at (-.5,0) {$\bullet$};
    \node[point] (3) at (.5,0) {$\bullet$};
    \node[point] (4) at (1.5,0) {$\bullet$};
    \node [above] at (1) {\tiny $1$};
    \node [above] at (2) {\tiny $2$};
    \node [above] at (3) {\tiny $3$};
    \node [above] at (4) {\tiny $4$};
    \draw [green!60!black, semithick] (1) to (2);
    \draw [magenta, semithick] (2) -- (3);
    \draw [blue, semithick] (3) -- (4);
  \end{tikzpicture}
  \right) =
  \begin{tikzpicture}[anchorbase,scale=.6]
    \tikzstyle{point}=[inner sep=0em]          
    \draw[dashed] (0,0) ellipse (2.5cm and 1.5cm);    
    \node[point] (1) at (-1.5,0) {$\bullet$};
    \node[point] (2) at (-.5,0) {$\bullet$};
    \node[point] (3) at (.5,0) {$\bullet$};
    \node[point] (4) at (1.5,0) {$\bullet$};
    \node [above] at (1) {\tiny $1$};
    \node [above] at (2) {\tiny $2$};
    \node [below] at (3) {\tiny $3$};
    \node [below] at (4) {\tiny $4$};
    \draw [green!60!black, semithick] (1) to [out=-40,in=-140] (3);
    \draw [magenta, semithick] (2) -- (3);
    \draw [blue, semithick] (2) to [out=40,in=140] (4);
  \end{tikzpicture}
  \]
  \caption{The action of the braid element \(\sigma_2 = \tau_{(2,3)}\) on the basic curves}
  \label{fig:actiononcurves}
\end{figure}
Furthermore, to each curve \(\gamma\) we can assign the braid \(\tau_{\gamma}\), which is the anticlockwise half Dehn twist along the curve.

Let \(V_q\) denote the Burau representation (see~\cref{def:burau-rep}) of the chosen Coxeter type (\(A_{n-1}\) in this case).
There is a way (which we do not detail here, but is explained, e.g. in~\cite{KhS} and in~\cite{Bigelow_Burau}) to assign an element of $V_q$ to any curve.
The basic curves $(i,i+1)$ are sent to $\alpha_i$.
For more complicated curves, the coefficient of $\alpha_i$ roughly counts how many times the curve navigates between points $i$ and $i+1$, by computing minimal intersection numbers with certain basic curves. 
Furthermore, each such intersection number is attached to a power of \(q\) based on the total winding number of the curve around the outermost disk from the starting point up to the current point.
Thus the coefficient of each \(\alpha_i\) is a certain graded intersection number.
This process is compatible with the pairing, in the following sense: it sends a graded version of the number of intersections between curves to the pairing.

Bigelow's criterion to disprove faithfulness~\cite[Theorem 1.4]{Bigelow_Burau} then expresses as follows: the representation is faithful if and only if there does not exist a pair of curves $\gamma_1$ and $\gamma_2$ that:
\begin{itemize}
\item do intersect non-trivially;
\item but have trivial pairing in $V_q$: $\langle \gamma_1,\gamma_2\rangle=0$.
\end{itemize}

Indeed if such a pair can be found, then the commutator of the corresponding half Dehn twists $[\tau_{\gamma_1},\tau_{\gamma_2}]$, is a non-trivial element in the kernel.
Reciprocally, if $\beta$ is in the kernel, then a pair of the kind $((i,i+1),\beta((i+2,i+3)))$ will meet the above criteria.

\begin{remark}\label{rem:general-strategy}
  The curve-based strategy does not immediately extend outside of type $A$, where one lacks geometric models.
  The key idea behind this paper is that the Khovanov--Seidel categorical representation~\cite{KhS} provides a natural replacement for curves, namely spherical objects in a certain category.
  This categorical representation is computable via the code in~\cite{BD_git}, and thus we are able to run computer searches for counterexamples.
\end{remark}

We will develop this idea and its consequences in Sections~\ref{sec:A3tilde} and~\ref{sec:strategyZ}. In particular, we state two versions of Bigelow's criterion for curves to generate a non-trivial element of the kernel (\cref{lem:CatBigelow1} and \cref{lem:CatBigelow2}), where the non-triviality uses information about the space of morphisms at the categorical level.

Finally, in Section~\ref{sec:bucket}, we find another condition for a curve-like object to generate a non-trivial element of the kernel (see \cref{lem:samecurve}). We mix this with Gibson, Williamson and Yacobi's bucket search~\cite{GibsonWilliamsonYacobi} to explore faithfulness in type $D_4$.

The three conditions provide variations of Bigelow's criterion, and extend it in the situations where geometric arguments are not readily available. Here are the three basic ideas we use, expressed in the classical curve context: for curves $\gamma_1$ and $\gamma_2$ and their associated half Dehn twists $\tau_{\gamma_1}$ and $\tau_{\gamma_2}$, 
\begin{itemize}
\item in~\cref{lem:CatBigelow1}: if $\langle \gamma_1,\gamma_2\rangle=0$, then $[\tau_{\gamma_1},\tau_{\gamma_2}]$ is in the kernel of the Burau representation;
\item in~\cref{lem:CatBigelow2}: if $\langle \gamma_1,\gamma_2\rangle=\pm q^{r}$, then $\tau_{\gamma_1}\tau_{\gamma_2}\tau_{\gamma_1}\tau_{\gamma_2}^{-1}\tau_{\gamma_1}^{-1}\tau_{\gamma_2}^{-1}$ is in the kernel of the Burau representation;
\item in~\cref{cor:Bigelow3}: if $\gamma_1$ and $\gamma_2$ are equal in $V_q$, then $\tau_{\gamma_1}\tau_{\gamma_2}^{-1}$ is in the kernel of the Burau representation.
\end{itemize}
Outside of type $A$, and in particular in the case of highly connected graphs, \cref{lem:CatBigelow2} might prove more useful.
For example in the case of the complete graphs, no two simple roots pair together to zero, which makes the original criterion harder to use.

\section{The affine \(A_3\) case} \label{sec:A3tilde}

Here is our first main result.

\begin{theorem}\label{thm:unfaithful-affine-a3}
  The Burau representation is not faithful in type $\tilde{A_3}$.
\end{theorem}

\begin{proof}
  Consider the following braids:
  \begin{align*}
    a & = \sigma_3^2\sigma_4\sigma_3\sigma_2\sigma_1\sigma_3^{-1}\sigma_4\sigma_3\sigma_2\sigma_1^{-2}\sigma_4,\\
    b & = \sigma_1^2\sigma_2^{-1}\sigma_4\sigma_1\sigma_3^{-1}\sigma_2\sigma_4^{-1}\sigma_3\sigma_1\sigma_4\sigma_1\sigma_2^{-1}\sigma_4^{-2}\sigma_3,\\
    \alpha &=a\sigma_3 a^{-1}, \\
    \beta &=b\sigma_2b^{-1}.
  \end{align*}
  Then we claim that $[\alpha,\beta]$ is a non-trivial element in the kernel of the Burau representation.

  Using Bigelow's trick, one can prove that \([\alpha,\beta]\) is indeed in the kernel by checking that the pairing between $ae_3$ and $be_2$ is zero, where $e_2$ and $e_3$ are the second and third basis vector.
  Non-triviality follows from the fact that the $\Hom$ space between the two corresponding spherical objects $aP_3$ and $bP_2$ is non-trivial.
\end{proof}

\begin{remark}
  Our original computer calculations output the following braid words for \(a\) and \(b\) respectively:
  \begin{align*}
    a' &=\sigma_3\sigma_1\sigma_2\sigma_1\sigma_{3}^{-1}\sigma_4\sigma_2\sigma_3\sigma_2\sigma_{3}^{-1}\sigma_1^{-2}\sigma_{4} \\
    b' &=\sigma_1\sigma_{4}^{-1}\sigma_1^2\sigma_3^{-2}\sigma_{2}^{-1}\sigma_4\sigma_1\sigma_{3}^{-1}\sigma_2\sigma_4^{-1}\sigma_3\sigma_{1}\sigma_{4}\sigma_{1}\sigma_2^{-1}\sigma_4^{-2}\sigma_3.
  \end{align*}
  We can transform the pair \((a',b')\) into the pair \((a,b)\) used in~\cref{thm:unfaithful-affine-a3} by braid moves, simultaneous letter cancellation, as well as by moving some letters from the front of \(b'\) to the front of \(a'\).
  Explicitly, we follow the following steps.
  \begin{enumerate}
  \item Reduce the sub-expression \(\sigma_2\sigma_3\sigma_2\sigma_3^{-1}\) appearing in \(a'\) to \(\sigma_3\sigma_2\).
  \item Swap \(\sigma_3\sigma_1\) at the start of \(a'\) to \(\sigma_1\sigma_3\), and then cancel \(\sigma_1\) from the front of both words.
  \item Swap \(\sigma_1^2\sigma_3^{-2}\) to \(\sigma_3^{-2}\sigma_1^2\) in \(b'\), and then move the expression \(\sigma_4^{-1}\sigma_3^{-2}\) appearing in the front of \(b'\) to the front of \(a\) as the inverse expression \(\sigma_3^2\sigma_4\).
  \end{enumerate}
  We make this transformation because the pair \((a,b)\) is much simpler, and produces a cleaner drawing as illustrated below.
\end{remark}

Although we used the Khovanov--Seidel categorical representation to find this counterexample, the case of affine type \(A\) does, in fact, have a geometric model (see for instance~\cite{GTW}, especially Sections 4 and 5).
The braid group of affine type \(A_n\) is a subgroup of the mapping class group of an annulus with \(n+1\) punctures.
We consider the action of this braid group on curves that begin and end at two punctures, do not self-intersect, and do not intersect any of the punctures in their interior.
Thus our objects \(aP_3\) and \(bP_2\) can be represented by such curves as well.

It is convenient to draw the curves on the universal cover of the annulus with \(n+1 = 4\) punctures.
The universal cover is then an infinite punctured strip, where each puncture lifts to \(\mathbb{Z}\)-many copies of itself.
A curve on the annulus from point \(p\) to \(q\) can be lifted uniquely to a curve starting at any lift of \(p\); it will then end at some lift of \(q\).

Label the four punctures in the annulus successively as \(1,2,3,4\).
Then \(P_1\), \(P_2\), and \(P_3\) are represented by the shortest curves joining \(1 \leftrightarrow 2\), \(2 \leftrightarrow 3\), and \(3 \leftrightarrow 4\) respectively, while \(P_4\) is the curve joining \(4 \leftrightarrow 1\) going around the hole of the annulus, as shown in~\cref{fig:curves-on-annulus}.
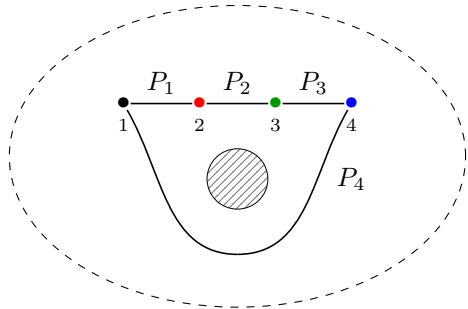
\begin{figure}[ht]
  \centering
  \begin{tikzpicture}
    \tikzstyle{pointlabel}=[font=\scriptsize]    
    \node [black, inner sep=0em] (1) at (1,0) {$\bullet$};
    \node [pointlabel, below=0em of 1] {\(1\)};
    \node [red, inner sep=0em] (2) at (2,0) {$\bullet$};
    \node [pointlabel, below=0 of 2] {\(2\)};    
    \node [green!60!black, inner sep=0em] (3) at (3,0) {$\bullet$};
    \node [pointlabel, below=0em of 3] {\(3\)};        
    \node [blue, inner sep=0em] (4) at (4,0) {$\bullet$};
    \node [pointlabel, below=0em of 4] {\(4\)};
    \draw[very thin, pattern=north east lines, pattern color=gray] (2.5, -1) circle (0.4);
    \path[semithick]
    (1) edge node[above] {\(P_1\)} (2)
    (2) edge node[above] {\(P_2\)} (3)
    (3) edge node[above] {\(P_3\)} (4)
    (4) edge[out=-120, in=0] (2.5, -2)
    (2.5,-2) edge[out=180, in=-60] (1)
    ;
\node at (4,-1) {\(P_4\)};
    \draw[dashed] (2.5,-0.7) ellipse (3cm and 2cm);
  \end{tikzpicture}
  \caption[Basic curves on the annulus]{The basic curves \(P_1\), \(P_2\), \(P_3\), and \(P_4\) as drawn on the annulus.}
  \label{fig:curves-on-annulus}
\end{figure}

In the universal cover, the curve \(P_i\) is then represented by the straight line joining any lift of the point \(i\) to the closest lift of the point \(i+1\) (modulo \(4\)).
The Dehn twist in a curve on the annulus is performed in the universal cover simultaneously around every lift of the curve.
With this correspondence, our curves \(aP_3\) and \(b P_2\) are drawn on the universal cover as shown in~\cref{fig:ap3,fig:bp2}.
It is evident from the pictures that the curves intersect: this is easiest to verify close to the left hand endpoint of the top curve, and also the right hand endpoint of the bottom curve.

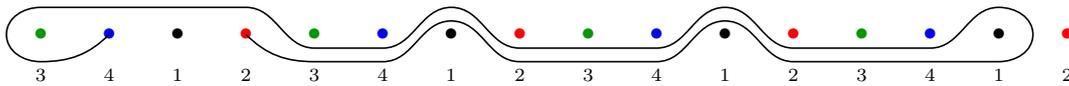
\begin{figure}[h]
  \centering
  \begin{tikzpicture}[anchorbase,scale=.9]
    \tikzstyle{pointlabel}=[font=\scriptsize]
    \tikzstyle{point}=[inner sep=0em]
    \node [point, green!60!black] (1) at (1,0) {$\bullet$};
    \node [pointlabel, below=0.7em of 1] {\(3\)};                        
    \node [point, blue] (2) at (2,0) {$\bullet$};
    \node [pointlabel, below=0.7em of 2] {\(4\)};                            
    \node [point, black] (3) at (3,0) {$\bullet$};
    \node [pointlabel, below=0.7em of 3] {\(1\)};                        
    \node [point, red] (4) at (4,0) {$\bullet$};
    \node [pointlabel, below=0.7em of 4] {\(2\)};                        
    \node [point, green!60!black] (5) at (5,0) {$\bullet$};
    \node [pointlabel, below=0.7em of 5] {\(3\)};                        
    \node [point, blue] (6) at (6,0) {$\bullet$};
    \node [pointlabel, below=0.7em of 6] {\(4\)};                        
    \node [point, black] (7) at (7,0) {$\bullet$};
    \node [pointlabel, below=0.7em of 7] {\(1\)};                        
    \node [point, red] (8) at (8,0) {$\bullet$};
    \node [pointlabel, below=0.7em of 8] {\(2\)};                        
    \node [point, green!60!black] (9) at (9,0) {$\bullet$};
    \node [pointlabel, below=0.7em of 9] {\(3\)};                        
    \node [point, blue] (10) at (10,0) {$\bullet$};
    \node [pointlabel, below=0.7em of 10] {\(4\)};                        
    \node [point, black] (11) at (11,0) {$\bullet$};
    \node [pointlabel, below=0.7em of 11] {\(1\)};                        
    \node [point, red] (12) at (12,0) {$\bullet$};
    \node [pointlabel, below=0.7em of 12] {\(2\)};                        
    \node [point, green!60!black] (13) at (13,0) {$\bullet$};
    \node [pointlabel, below=0.7em of 13] {\(3\)};                        
    \node [point, blue] (14) at (14,0) {$\bullet$};
    \node [pointlabel, below=0.7em of 14] {\(4\)};                        
    \node [point, black] (15) at (15,0) {$\bullet$};
    \node [pointlabel, below=0.7em of 15] {\(1\)};                        
    \node [point, red] (16) at (16,0) {$\bullet$};
    \node [pointlabel, below=0.7em of 16] {\(2\)};
    \draw[semithick] (2.center) to [out=225, in=0] (1,-0.4)
    to [out=180, in=-90] (0.5,0)
    to [out=90, in=180] (1,0.4)
    to (4, 0.4)
    to [out=0, in=180] (5, -0.2)
    to (6, -0.2)
    to [out=0, in=180] (7, 0.4)
    to [out=0, in=180] (8, -0.2)
    to (10, -0.2)
    to [out=0, in=180] (11, 0.4)
    to [out=0, in=180] (12, -0.2)
    to (14, -0.2)
    to [out=0, in=180] (15, 0.4)
    to [out=0, in=90] (15.5, 0)
    to [out=-90, in=0] (15, -0.4)
    to (12, -0.4)
    to [out=180, in=0] (11, 0.2)
    to [out=180, in=0] (10, -0.4)
    to (8, -0.4)
    to [out=180, in=0] (7, 0.2)
    to [out=180, in=0] (6, -0.4)
    to (5, -0.4)
    to [out=180, in=-45] (4.center)
    ;
  \end{tikzpicture}
  \caption[The curve \(aP_3\).]{The curve \(aP_3\) drawn in the universal cover of the annulus with four marked points.}\label{fig:ap3}
\end{figure}

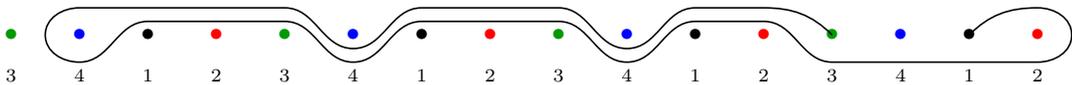
\begin{figure}[h]
  \centering
  \begin{tikzpicture}[anchorbase,scale=.9, trim left=5pt]
    \tikzstyle{pointlabel}=[font=\scriptsize]
    \tikzstyle{point}=[inner sep=0em]
    \node [point, green!60!black] (1) at (1,0) {$\bullet$};
    \node [pointlabel, below=0.7em of 1] {\(3\)};                        
    \node [point, blue] (2) at (2,0) {$\bullet$};
    \node [pointlabel, below=0.7em of 2] {\(4\)};                            
    \node [point, black] (3) at (3,0) {$\bullet$};
    \node [pointlabel, below=0.7em of 3] {\(1\)};                        
    \node [point, red] (4) at (4,0) {$\bullet$};
    \node [pointlabel, below=0.7em of 4] {\(2\)};                        
    \node [point, green!60!black] (5) at (5,0) {$\bullet$};
    \node [pointlabel, below=0.7em of 5] {\(3\)};                        
    \node [point, blue] (6) at (6,0) {$\bullet$};
    \node [pointlabel, below=0.7em of 6] {\(4\)};                        
    \node [point, black] (7) at (7,0) {$\bullet$};
    \node [pointlabel, below=0.7em of 7] {\(1\)};                        
    \node [point, red] (8) at (8,0) {$\bullet$};
    \node [pointlabel, below=0.7em of 8] {\(2\)};                        
    \node [point, green!60!black] (9) at (9,0) {$\bullet$};
    \node [pointlabel, below=0.7em of 9] {\(3\)};                        
    \node [point, blue] (10) at (10,0) {$\bullet$};
    \node [pointlabel, below=0.7em of 10] {\(4\)};                        
    \node [point, black] (11) at (11,0) {$\bullet$};
    \node [pointlabel, below=0.7em of 11] {\(1\)};                        
    \node [point, red] (12) at (12,0) {$\bullet$};
    \node [pointlabel, below=0.7em of 12] {\(2\)};                        
    \node [point, green!60!black] (13) at (13,0) {$\bullet$};
    \node [pointlabel, below=0.7em of 13] {\(3\)};                        
    \node [point, blue] (14) at (14,0) {$\bullet$};
    \node [pointlabel, below=0.7em of 14] {\(4\)};                        
    \node [point, black] (15) at (15,0) {$\bullet$};
    \node [pointlabel, below=0.7em of 15] {\(1\)};                        
    \node [point, red] (16) at (16,0) {$\bullet$};
    \node [pointlabel, below=0.7em of 16] {\(2\)};

    \draw[semithick] (15.center) to [out=45, in=180] (16, 0.4);
    \draw[semithick] (16, 0.4) to [out=0, in=90] (16.5, 0)
    to [out=-90, in=0] (16,-0.4)
    to (13,-0.4)
    to [out=180, in=0] (12, 0.2)
    to (11, 0.2)
    to [out=180, in=0] (10, -0.4)
    to [out=180, in=0] (9,  0.2)
    to (7,  0.2)
    to [out=180, in=0] (6, -0.4)
    to [out=180, in=0] (5,  0.2)
    to (3,  0.2)
    to [out=180, in=0] (2, -0.4)
    to [out=180, in=-90] (1.5, 0)
    to [out=90, in=180] (2, 0.4)
    to (5, 0.4)
    to [out=0, in=180] (6, -0.2)
    to [out=0, in=180] (7, 0.4)
    to (9, 0.4)
    to [out=0, in=180] (10, -0.2)
    to [out=0, in=180] (11, 0.4)
    to (12, 0.4)
    to [out=0, in=135] (13.center)
    ;
  \end{tikzpicture}
  \caption[The curve \(bP_2\).]{The curve \(bP_2\) drawn in the universal cover of the annulus with four marked points.}\label{fig:bp2}
\end{figure}

\begin{remark}
  It is a classical result that Artin--Tits groups in type $\tilde{A_n}$ embed in braid groups of type $A_{n+1}$: the former can be represented by braids in an annulus, the core of which can be made into a fixed strand of a classical braid (see for example~\cite{Allcock}).
  First, this embedding suggests a curve-based approach to the above result.
  We did not use such an approach computationally, since we wanted to investigate more cases than those accessible by curve-based methods.
  Second, this embedding raises the question of whether the lack of faithfulness claimed above is a consequence of the lack of faithfulness in type $A_4$. There do not seem to be direct arguments in favor of this however, as the Burau pairing in affine type $A$ is not the one inherited by the inclusion of groups.
\end{remark}

\section{Strategy and computations over $\Z$} \label{sec:strategyZ}

In this section we expand on~\cref{rem:general-strategy} to explain the design of the search that led to \cref{thm:unfaithful-affine-a3}.
Our strategy is based on the one used by Bigelow, with curves replaced by spherical objects in Khovanov--Seidel's category.
It should be noted that Khovanov--Seidel's representation can be easily defined in any simply-laced type, but that no faithfulness result is known outside of types $A$, $D$, $E$ and $\tilde{A}$.

We recall some important features of the Khovanov--Seidel categorical representation, referring the reader to~\cite{KhS} or the survey~\cite{Queffelec_KhS} for further details.
Fix a Coxeter diagram \(\Gamma\).
Let \(A_{\Gamma}\) be the zigzag algebra of \(\Gamma\) (see, e.g.,~\cite[Section 5]{hue.kho:01} as well as~\cite[Remark 6.6]{BDL}).
Let \(\mathcal{C}_{\Gamma}\) be the bounded homotopy category of finitely-generated graded projective modules over \(A_{\Gamma}\).
We set \(\{P_i \mid i \in V(\Gamma)\}\) to be the projective indecomposable \(A_{\Gamma}\)-modules.
Following~\cite{KhS}, we can define an action of the Artin--Tits group \(\mathbf{B}(\Gamma)\) on \(\mathcal{C}_{\Gamma}\).
The resulting action of \(\mathbf{B}(\Gamma)\) on the Grothendieck group \(K_0(\mathcal{C}_{\Gamma})\) recovers the Burau representation \(V_q = V_{q, \Gamma}\).

The morphisms in \(\mathcal{C}_{\Gamma}\) are bigraded by internal graded degree and homological degree.
We write \(\HOM^{g,h}(X,Y)\) to mean the morphisms from \(X\) to \(Y\) of graded degree \(g\) and homological degree \(h\).
That is,
\[\HOM^{g,h}(X,Y) = \HOM(X,Y\{g\}[h]),\]
where the functors \(\{g\}, [h] \colon \mathcal{C}_{\Gamma} \to \mathcal{C}_{\Gamma}\) denote the grading shift and homological shift by \(g\) and \(h\) respectively.
We write \(\HOM^{\bullet,\bullet}(X,Y)\) to mean the total bigraded hom space.
That is,
\[\HOM^{\bullet,\bullet}(X,Y) = \bigoplus_{g,h \in \mathbb{Z}}\HOM^{g,h}(X,Y).\]
The spaces of bigraded morphisms enjoy the following duality for any \(g, h \in \mathbb{Z}\) and objects \(X, Y\) of \(\mathcal{C}_{\Gamma}\):
\[\HOM^{g,h}(X,Y) \cong \HOM^{2-g, -h}(Y,X)^{\vee}.\]
For \(X,Y \in \mathcal{C}_{\Gamma}\), the pairing on $K_0(\mathcal{C}_\Gamma) = V_q$ between \([X]\) and \([Y]\) is simply the \((q,-1)\) Euler characteristic of the total hom space.
That is:
\[\langle [X], [Y] \rangle = \sum_{g \in \mathbb{Z}}q^g \left(\sum_{h \in \mathbb{Z}}(-1)^h \dim \HOM^{g,h}(X,Y)\right).\]
Note that by the duality between the hom-spaces, the pairing \(\langle [X],[Y] \rangle\) is related to \(\langle [Y], [X] \rangle\).
More precisely, suppose that \(\langle [X],[Y] \rangle = p(q)\) for some polynomial expression in \(\mathbb{Z}[q,q^{-1}]\).
Then 
\(\langle [Y], [X] \rangle = q^2 p(q^{-1})\).

A special role is played by the \emph{(2-)spherical objects} in \(\mathcal{C}_{\Gamma}\): these are the objects \(X\) such that:
\begin{itemize}
\item \(\END^{\bullet,h}(X) = 0\) for \(h \neq 0\); and
\item the algebra \(\END^{\bullet, 0}(X)\) is isomorphic to the cohomology algebra of the \(2\)-sphere.
\end{itemize}
Each spherical object \(X \in \mathcal{C}_{\Gamma}\) gives rise to an autoequivalence \(\tau_X \colon \mathcal{C}_{\Gamma} \xrightarrow{\sim} \mathcal{C}_{\Gamma}\), called the \emph{spherical twist} with respect to \(X\).
In particular, the generating projective modules \(\{P_i \mid i \in V(\Gamma)\}\) are always spherical, and the associated twists \(\tau_{P_i}\) satisfy the definining relations of the group \(\mathbf{B}(\Gamma)\) up to equivalence.

In the situations when we have geometric models such as in type \(A\), the spherical objects (up to shift) are the ones that can be recovered from curves, as explained in~\cref{subsec:proof-strategies}.
Recall also that the pairing between two spherical objects is realised by the graded intersection number between the corresponding curves.
Moreover, the spherical twist in a spherical object corresponds to a half Dehn twist in the corresponding curve.
For these reasons, in the general situation, we regard the spherical objects of \(\mathcal{C}_{\Gamma}\) as analogues of curves.
We now explain how we adapt Bigelow's strategy to these more general situations.

We now describe two versions of Bigelow's criterion, each of which is useful depending on the situation.
The first criterion (\cref{lem:CatBigelow1}) is very similar to the one used by Bigelow~\cite[Theorem 1.4]{Bigelow_Burau}, while the second one (\cref{lem:CatBigelow2}) is new.
Furthermore, the proof of the first one is similar to, but easier than the proof of the second one. 
We thus only prove the second criterion here, and leave the proof of the first criterion to the reader.
We emphasize again that we are replacing the geometric condition in Bigelow's criterion by a homological one, in which curves are replaced by spherical objects.

\begin{lemma}\label{lem:CatBigelow1}
  Assume we have two spherical objects $X_1$ and $X_2$ in \(\mathcal{C}_{\Gamma}\), so that:
  \begin{itemize}
  \item $\HOM^{\bullet,\bullet}(X_1,X_2)\neq \{0\}$; and 
  \item $\langle [X_1],[X_2]\rangle =0$ in $V_q$.
  \end{itemize}
  Then the commutator $[\tau_{X_1},\tau_{X_2}] = \tau_{X_1} \tau_{X_2} \tau_{X_1}^{-1} \tau_{X_2}^{-1}$ of their associated spherical twists is a non-trivial element of the kernel of the Burau representation.
\end{lemma}

\begin{lemma}\label{lem:CatBigelow2}
  Assume we have two spherical objects $X_1$ and $X_2$, so that:
  \begin{itemize}
  \item $\dim \HOM^{\bullet,\bullet}(X_1,X_2) > 1$; and 
  \item $\langle [X_1],[X_2]\rangle = q$ in $V_q$.
  \end{itemize}
  Then $\tau_{X_1}\tau_{X_2}\tau_{X_1}\tau_{X_2}^{-1}\tau_{X_1}^{-1}\tau_{X_2}^{-1}$ is a non-trivial element of the kernel of the Burau representation.
\end{lemma}

Before we prove this statement, recall from~\cite[Section 2d]{KhS} that applying the spherical twist in an object \(X\) is the functor of tensoring with a certain complex of bimodules over the zigzag algebra \(A_{\Gamma}\), namely
  \[\tau_X \leftrightarrow (X \otimes X^{\vee} \to A_{\Gamma}),\]
  where the map from \(X \otimes X^{\vee} \to A_{\Gamma}\) is simply multiplication.
  Henceforth, let us identify \(\tau_X\) with the complex \(X \otimes X^{\vee} \to A_{\Gamma}\), so that composition of twists turns into tensoring these complexes together over \(A_{\Gamma}\).
  Furthermore, by standard results we have a natural isomorphism of bigraded vector spaces
  \[X^{\vee} \otimes_{A_{\Gamma}} Y \cong \HOM^{\bullet,\bullet}(X,Y).\]
\begin{proof}[Proof of~\cref{lem:CatBigelow2}]
  The fact that the expression \(\tau_{X_1}\tau_{X_2}\tau_{X_1}\tau_{X_2}^{-1}\tau_{X_1}^{-1}\tau_{X_2}^{-1}\) lies in the kernel of the Burau representation is clear from the pairing.
  It remains to prove that it is non-trivial, which amounts to proving that:
  \[
    \tau_{X_1}\tau_{X_2}\tau_{X_1}\neq \tau_{X_2}\tau_{X_1}\tau_{X_2}.
  \]

  The spherical twists $\tau_{X_1}$ and $\tau_{X_2}$ can be expressed as follows:
  \begin{align*}
    \tau_{X_1} &= X_1\otimes X_1^\vee \rightarrow A_\Gamma \\
    \tau_{X_2} &= X_2\otimes X_2^\vee \rightarrow A_{\Gamma}.
  \end{align*}
  The composition $\tau_{X_1} \tau_{X_2} \tau_{X_1} = \tau_{X_1}\otimes_{A_{\Gamma}} \tau_{X_2}\otimes_{A_{\Gamma}} \tau_{X_1}$ is then the following complex of bimodules.
  \begin{center}
    \begin{tikzcd}[row sep=2em, column sep=1.5em]
      &X_1\otimes (X_1^\vee \otimes_{A_\Gamma} X_1)\otimes X_1^\vee\arrow{r} \arrow{dr}
      & X_1\otimes X_1^\vee\arrow{dr} \\
      X_1\otimes (X_1^\vee \otimes_{A_{\Gamma}} X_2)\otimes (X_2^\vee \otimes_{A_{\Gamma}} X_1)\otimes X_1^\vee \arrow{ur} \arrow{dr} \arrow{r}
      & X_2\otimes (X_2^\vee \otimes_{A_\Gamma} X_1)\otimes X_1^\vee \arrow{ur} \arrow{dr}
      & X_1\otimes X_1^\vee \arrow{r}
      & A_\Gamma\\
      & X_1\otimes (X_1^\vee\otimes_{A_\Gamma} X_2) \otimes X_2^\vee \arrow{ur} \arrow{r}
      & X_2\otimes X_2^\vee \arrow{ur}
    \end{tikzcd}
  \end{center}

  Recall that $X_1^\vee \otimes X_1\simeq \C\oplus \C\langle 2\rangle$ because \(X\) is spherical.
  By the assumption on the morphisms from \(X_1\) to \(X_2\) as well as the pairing \(\langle [X_1],[X_2] \rangle\), we have
  \[X_1^{\vee} \otimes_{A_{\Gamma}}X_2 \cong \mathbb{C}\langle 1 \rangle \oplus W_1\]
  for some non-trivial bigraded vector space $W_1$ that has trivial bigraded dimension (i.e. \((q,-1)\)-dimension as in the definition of the pairing).

  Observe by the duality that \(\dim\HOM^{\bullet,\bullet}(X,Y) > 1\) if and only if \(\dim\HOM^{\bullet,\bullet}(Y,X) > 1\).
  Also, \(\langle [X_1], [X_2] \rangle = q\) implies that \(\langle [X_2], [X_1] \rangle = q\) as well.
  So we also have that
  \[ X_2^\vee\otimes_{A_{\Gamma}} X_1\simeq \C\langle 1\rangle \oplus W_2\]
  for some non-trivial bigraded vector space $W_2$ that has trivial bigraded dimension.

  The leftmost position now becomes
  \begin{align*}
    X_1 \otimes (\mathbb{C}\langle1 \rangle \oplus W_1) \otimes (\mathbb{C}\langle 1 \rangle \otimes W_2) \otimes X_1^{\vee} = &(X_1 \otimes \mathbb{C}\langle 2 \rangle \otimes X_1^{\vee}) \oplus \\
    &(X_1 \otimes (W_1\langle 1 \rangle \oplus W_2\langle 1 \rangle \oplus (W_1 \otimes W_2)) \otimes X_1^{\vee}).
  \end{align*}
  The top-left position becomes
  \[
    X_1\otimes (X_1^\vee \otimes_{A_{\Gamma}} X_1)\otimes X_1^\vee \cong (X_1 \otimes \mathbb{C} \otimes X_1^{\vee}) \oplus (X_1 \otimes \mathbb{C}\langle 2 \rangle \otimes X_1^{\vee}).
  \]
  We can now apply Gaussian elimination of complexes (see, e.g.~\cite[Lemma 4.2]{bar:07}).
  The term \(X_1 \otimes \mathbb{C}\langle 2 \rangle \otimes X_1^{\vee}\) in the leftmost position cancels with the same term in the top-left position, while the term \(X_1 \otimes \mathbb{C} \otimes X_1^{\vee} = X_1 \otimes X_1^{\vee}\) in the top-left position cancels with the top-right position.
  After these cancellations, the top row disappears, and we obtain the following reduced complex.
  \begin{center}
    \begin{tikzcd}[row sep=2em, column sep=1.5em]
      X_1 \otimes (W_1\langle 1 \rangle \oplus W_2\langle 1 \rangle \oplus (W_1 \otimes W_2)) \otimes X_1^{\vee} \arrow{r} \arrow{dr}
      & X_2\otimes (\C\langle 1\rangle \oplus W_2)\otimes X_1^\vee \arrow{r} \arrow{dr}
      & X_1\otimes X_1^\vee \arrow{r}
      & A_\Gamma\\
      & X_1\otimes  (\C\langle 1\rangle \oplus W_1) \otimes X_2^\vee \arrow{r} \arrow{ur}
      & X_2\otimes X_2^\vee \arrow{ur}
    \end{tikzcd}
  \end{center}
In the new complex, all of the terms except for the leftmost one are symmetric under exchanging $1$ and $2$.
Setting \(U = W_1\langle 1 \rangle \oplus W_2 \langle 1 \rangle \oplus (W_1 \otimes W_2)\), we see that the left-most terms in $\tau_{X_1}\tau_{X_2}\tau_{X_1}$ and $\tau_{X_2} \tau_{X_1}\tau_{X_2}$ are, respectively, $X_1\otimes U \otimes X_1^\vee$ and \(X_2\otimes U \otimes X_2^\vee\).
Since $U$ is non-trivial and $X_1$ and $X_2$ are not isomorphic, these two terms are different.
We see that
\[\tau_{X_1}\tau_{X_2}\tau_{X_1} \neq \tau_{X_2} \tau_{X_1}\tau_{X_2}.\]
Hence the desired expression
\(\tau_{X_1}\tau_{X_2}\tau_{X_1}\tau_{X_2}^{-1}\tau_{X_1}^{-1}\tau_{X_2}^{-1}\)
is non-trivial, because it is sent to a non-trivial element under the Khovanov--Seidel representation.
\end{proof}

The second version of the criterion (\cref{lem:CatBigelow2}) is especially useful when the graph is highly connected, when there might not be many pairs of braids that commute.
It is interesting to notice that these two criteria combined will detect the fact that the kernel of the Burau representation is bigger than the kernel of the Khovanov--Seidel representation.
When the latter is faithful (in particular in spherical type and affine type \(A\)), this gives an effective criterion to detect the kernel of the Burau representation. This follows from the following lemma.

\begin{lemma} \label{lem:equivalence}
  Assume that a braid $\beta$ acts on $\mathcal{C}_{\Gamma}$ non-trivially, but is in the kernel of the Burau representation. Then there exist $(i,j)\in V(\Gamma)$ such that one of the following holds:
  \begin{itemize}
  \item $\langle \alpha_i,\alpha_j\rangle =0$ and $\dim(\HOM^{\bullet}(\beta P_i,P_j))>0$, in which case \cref{lem:CatBigelow1} applies; or 
  \item $\langle \alpha_i,\alpha_j\rangle=q$ and $\dim(\HOM^{\bullet}(\beta P_i,P_j))>1$, in which case \cref{lem:CatBigelow2} applies.
  \end{itemize}
\end{lemma}

Here by a non-trivial action, we mean that $\beta$ does not act as the composition of homological shifts on each block corresponding to each connected component of $\Gamma$.

\begin{proof}
  Since $\beta$ acts non-trivially on $\calC$, in particular there exists $i\in V(\Gamma)$ so that $\beta P_i\neq P_i$. But since $\beta$ is in the kernel of the Burau representation, $\beta(\alpha_i)=\alpha_i$, which is equivalent to writing the following in $K_0(\calC)$:
  \[
\beta([P_i])=[P_i] \text{ for every }i \in V(\Gamma).
\]
Now, filter the complex $\beta P_i$ by the path length grading, and consider the bottom piece (that comes with highest shift $b$ in grading). One sees a direct sum $\bigoplus_{k}P_{l_k}\{b\}[h_k]$, and $\beta P_i$ looks like:
\[
\beta P_i = \Cone\left(\bigoplus_k P_{l_k}[h_k]\rightarrow X'\right).
\]
with $X'$ made of $P_k$'s shifted by strictly less than $b$ in path length grading.

If for one of the $l_k$'s one has $\langle \alpha_i,\alpha_{l_k}\rangle=0$, then the identity map from $P_{l_k}$ in $X$ to $P_{l_k}$ will generate a non-zero element of $\HOM(X,P_{l_k})$. Thus the dimension of this $\HOM$-space is strictly positive and we are done.

If for one of the $l_k$'s one has $\langle \alpha_i,\alpha_{l_k}\rangle=q$, then again the identity map from $P_{l_k}$ in $X$ to $P_{l_k}$ will generate a non-zero element of $\Hom(X,P_{l_k})$. If $P_{l_k}$ appears more than once (that is, $\exists k'$ with $l_k=l_{k'}$), then the dimension is strictly greater than $2$ and we are done. Otherwise, this means that in $[\beta P_i]$, $q^{b}\alpha_{l_k}$ appears with coefficient $(-1)^{h_k}\neq 0$, which contradicts the fact that $[\beta P_i]=[P_i]$.

This leaves us with the situation where all $l_k=i$. In this case, we do the same process with the top slice, of path length degree $t<b$ (this time, surviving maps will be loop maps). The only case where we cannot conclude is when only $P_i$'s appear as well.

If $\beta P_i$ is not concentrated in a single $q$-degree, that is if $t\neq b$, then any identity map from a $P_i$ in degree $t$ to a $P_i$ in degree $b$ produces a non-zero map in homotopy. One at least of the top and bottom slices must contain two $P_i$'s in different homological grading: if not, both the top and bottom slices would produce non-zero terms in the Grothendieck group, which would contradict $[\beta P_i]=[P_i]$. This means that there is a non-zero map in odd homological degree in $\END^{\bullet,\bullet}(\beta P_i)$, which contradicts the fact that $\END^{\bullet,\bullet}(\beta P_i)\simeq\END^{\bullet,\bullet}(P_i)\simeq \C\oplus \C\langle 2\rangle$.

So the only case left is when $b=t$. Then since the complex is indecomposable, one must have $\beta P_i=P_i\{b\}[h]$, and since $[\beta P_i]=[P_i]$, this determines $b=0$ and $h\in 2\Z$.

The only case where we cannot conclude yet is when $\beta P_j=P_j[h_j]$ for each $j$, with $h_j\in 2\Z$. If $i$ and $j$ are adjacent in $\Gamma$, the fact that $\HOM(P_i,P_j)\simeq\HOM(\beta P_i, \beta P_j)$ lies in homological degree zero implies that $h_i=h_j$. By connectedness, this implies that $h_i=h_{j}$ if $i$ and $j$ belong to the same connected component of $\Gamma$. But then this means that $\beta$ acts trivially, so this case does not happen.
\end{proof}

Our strategy of search goes as follows.
\begin{enumerate}
\item We produce a large set (size approximately 2 million) of $q$-polynomials for generalized curves, simply by acting on elementary curves with generators of the braid group.
  We usually classify these curves by their associated root (namely the value of the \(q\)-polynomial at $q=1$, possibly modulo 2 to get a finite list) and a length criterion, for instance the sum of absolute values of coefficients of the $q$-polynomial.
\item We then record the pairs of \(q\)-polynomials $(\gamma_1,\gamma_2)$ that have trivial pairing (for one or the other version of the criterion).
  We usually restrict the search to specific pairs of roots to limit the size of the calculation, and also try to avoid duplicates.
  For example, the pair $(\beta \gamma_1,\beta \gamma_2)$ would be redundant for any braid \(\beta\), and we try to rule out such cases.
\item Finally, we compute the objects in the category that correspond to $\gamma_1$ and $\gamma_2$ and compute the $\HOM^{\bullet}$-space.
  If this is non-trivial (depending on which version of the criterion we are using), then we have a counterexample.
\end{enumerate}
Our code is available in the Git repository~\cite{BQ_git}.

\section{Random searches over $\Z/p\Z$}\label{sec:bucket}

In finite type (so, in our case, $A_3$ and $D_4$), we adapt the strategy of Gibson, Williamson and Yacobi from~\cite{GibsonWilliamsonYacobi} to a curve-search random algorithm.
Recall that these authors found an explicit braid in the kernel of the Burau representation over $F_5=\Z/5\Z$ by studying the projective length (or \textsf{projlen}) of braids, which is the difference between the top and bottom $q$-degree of the terms of a matrix.
We translate this search in terms of curves and move it from the classical Garside structure to the dual one.

Recall that braid groups of finite type admit a so-called \emph{dual} set of generators, due to Birman-Ko-Lee~\cite{BKL} and Bessis~\cite{Bessis}.
Let \(n\) be the rank of the chosen type.
Fix a Coxeter element $\gamma$, which is the product of the generators $\sigma_i$ in some fixed chosen order.
In practice, we usually take the product in the order given by the index; for instance, $\gamma=\sigma_1\sigma_2\sigma_3\sigma_4$ for $D_4$.
Then consider the following subset of \(B(\Gamma)\):
\[
\mathbb{T} =\{\beta \sigma_i\beta^{-1},\;\beta \in \B(\Gamma)\}.
\]

\begin{definition}
  Define $[1,\gamma]$ to be the set of braids:
  \[
[1,\gamma]:=  \{\tau_1\cdots \tau_k \mid \text{ there exist } \tau_{k+1},\dots,\tau_n \text{ in }\mathbb{T} \text{ such that } \tau_1\cdots \tau_n=\gamma
  \}
  \]
\end{definition}
In other words, we consider prefixes of $\gamma$ written as a minimal length product of lifts of reflections.
It follows from Bessis' work that this definition is meaningful.
Atoms in this set are $\mathbb{T}\cap [1,\gamma]$.

For $\beta_1, \beta_2\in [1,\gamma]$, we will say that \(\beta_1\) \emph{divides} \(\beta_2\), written as $\beta_1 \mid \beta_2$, if:
\begin{align*}
\text{there exists } &\beta_3\in [1,\gamma]\text{ such that } \beta_2=\beta_1\beta_3, \text{ or equivalently, if:}\\
\text{there exists } &\beta_4\in [1,\gamma]\text{ such that } \beta_2=\beta_4\beta_1.
\end{align*}
That is, left and right divisibility are the same.

Associated to the set of generators $[1,\gamma]$ is a normal form (Garside normal form), which in our case allows us to generate non-trivial braids of increasing complexity.
We will first tweak this normal form for it to be adapted to the curve search.
Before doing so, let us slightly change the Burau representation to an equivalent form that is more suited to these new generators. Recall that our choice of Coxeter element \(\gamma\) corresponds to a fixed order on the Artin generators; in this case, by increasing index.
Now, we restrict to finite type, and define the following:
\[
\langle \alpha_i,\alpha_j\rangle_d =\begin{cases}
1+q \;\text{if}\;i=j \\
1\; \text{if}\;i\leq j\;\text{and}\; m_{ij}=3 \\
q\; \text{if}\;j\leq i\;\text{and}\; m_{ij}=3.
\end{cases}
\]

The explicit action of generator thus becomes, for $\alpha \in V_q$:
\begin{align} \label{eq:dualaction}
  \sigma_i(\alpha)&=\alpha-\langle \alpha_i,\alpha\rangle_d\alpha_i \\
  \sigma_i^{-1}(\alpha)&=\alpha-q^{-1}\langle \alpha,\alpha_i\rangle_d\alpha_i. \nonumber
\end{align}

We define the \emph{spread} of a matrix as the difference between the top $q$-degree of all its entries, and the bottom $q$-degree of all of its entries. A nice feature of this new form is that images of elements of $[1,\gamma]$ only have degree $0$ and $1$, and thus spread $0$ (only the $\id$ and $\gamma$) or $1$.

It is proven in \cite{LQ} that the categorical spread controls the Garside length, which is equivalent to the categorical faithfulness of the Burau representation. In this case, the goal is to find braids with non-zero Garside length but spread $0$.
We can obtain all braids from braids in the positive monoid by multiplication by a power of $\gamma$.
Since the action of \(\gamma\) is entirely controlled, we restrict the search to braids in the positive monoid.
In~\cite{GibsonWilliamsonYacobi}, such a search strategy is used with the classical Garside structure, where braids are sampled in buckets indexed by the Garside length and the projective length (the spread of the matrices with the $q$-conventions from the beginning of the paper).

Our goal is to merge this search strategy with the ones based on curves inspired by Bigelow. This stands on the following lemma.

\begin{lemma} \label{lem:samecurve}
  Let \(\Gamma\) be of simply-laced finite type.
  Fix a standard generator $\sigma_i\in B(\Gamma)$.
  Consider a writing $\beta=\beta_n\cdots \beta_1$ such that:
  \begin{itemize}
  \item the writing is a right-greedy normal form for $\beta$, that is: $\beta_i\in [1,\gamma]$ and for every \(\mu \in \mathbb{T}\) such that $\mu \mid \beta_i$, we have $\mu\beta_{i-1}\notin [1,\gamma]$;
  \item $\beta\sigma_i$ is also a right-greedy form for its associated braid element;
  \item $\sigma_i$ does not divide $\beta_1$.
  \end{itemize}
Then $[\beta,\sigma_i]\neq \id$.
\end{lemma}
\begin{proof}
  Writing $[\beta,\sigma_i]=\beta\sigma_i\beta^{-1}\sigma_i^{-1}$ already gives a reduced minimal decomposition.
\end{proof}

It should be emphasized that this writing for $\beta$ is precisely designed to ensure that the categorical spread equals the Garside length: this follows from~\cite[Prop 4.6 and Cor 4.10]{LQ}. Then the non-triviality is stated in Proposition 4.11. We will use the previous lemma to find elements in the kernel of the Burau representation with coefficients in rings $F=\Z/p\Z[q,q^{-1}]$, as stated below.

\begin{corollary} \label{cor:Bigelow3}
Let $F$ be a unital ring, and $\phi$ a unital ring homomorphism from $\Z[q,q^{-1}]$ to $F$. Assume that $\sigma_i$ and $\beta$ are as above, and furthermore that $\phi(\beta)\alpha_i=q^l\alpha_i$ (in $F\otimes_{\phi}V_q$). Then $[\beta,\sigma_i]$ is a non-trivial element in the kernel of the Burau representation with coefficients in $F$.
\end{corollary}

\begin{proof}
  The commutator $[\beta,\sigma]=(\beta\sigma_i\beta^{-1})\sigma_{i}^{-1}$ is a product of spherical twists over spherical objects that are not distinguished by the Burau representation with coefficients in $F$. We will check that the induced Burau action is trivial. The fact that the braid is non-trivial is a consequence of~\cref{lem:samecurve}.

  We will use an easy generalization of~\cref{eq:dualaction}: for $\tau=\beta\sigma_i\beta^{-1}$ and $\alpha\in V_q$, we have:
  \begin{align*}
    \tau(\alpha)&=\alpha-\langle\beta \alpha_i,\alpha_d(\beta\alpha_i) \\
    \tau^{-1}(\alpha)&=\alpha-q^{-1}\langle\alpha,\beta\alpha_i\rangle_d (\beta\alpha_i)
  \end{align*}
  The point of these formulas is to emphasize the role of a curve played by $\beta\alpha_i$ with respect to the twist $\tau=\beta\sigma_i\beta^{-1}$.
  
Write $[\beta,\sigma_i]=(\beta\sigma_i\beta^{-1})\sigma_i^{-1}$ and focus on $\beta\sigma_i\beta^{-1}$, whose action is thus given by:
  \[
  \beta\sigma_i\beta^{-1}(\alpha)=\alpha-\langle\beta \alpha_i,\alpha\rangle_d(\beta\alpha_i).
  \]
  Recall that $\beta\alpha_i=q^{l}\alpha_i$, so that:
  \begin{align*}
      \beta\sigma_i\beta^{-1}(\alpha)&=\alpha-q^{-l}\langle\alpha_i,\alpha\rangle_dq^l\alpha_i=\alpha-\langle \alpha_i,\alpha\rangle_d\alpha_i.
  \end{align*}
  This is the same action as $\sigma_i$, and thus $(\beta\sigma_i\beta^{-1})\sigma_i^{-1}$ has trivial Burau action.
\end{proof}

This idea yielded several unfaithfulness results in type $D_4$. Note that unfaithfulness over $F_p$ for $p=2$, $3$ or $5$ follows from type $A$ embedding.
\begin{theorem}\label{thm:d4-unfaithful-finite-rings}
  The Burau representation of type $D_4$ is not faithful over $\Z/p\Z$ with $p\leq 16$.
\end{theorem}
\begin{proof}
  Take $\sigma_1=\sigma_1$ and $\beta$ as follows. Rather than using the classical generators $\sigma_i$, in the interest of space, we will only provide the sequence of indices. $i$ stands for $\sigma_i$, while $-i$ stands for $\sigma_i^{-1}$.

  For $p=6$, we take:
  \begin{align*}
    [&\scriptstyle -2, -1, 2, -4, -3, 1, -2, -1, 3, 4, -2, -1, -3, -4, 2, -3, -2, 2, -4, -2, -4, -3, 1, -2, -1, 3, 4, -1, -3, -4, -3, -4, -2, 4, 3, -2, 1, 2, -3, 
      \\&
      \scriptstyle
      -2, -1, 1, 2, -4, -2, -1, -3, -4, -2, 4, 3, 2, -4, -3, 1, -2, -1, 3, 4, -2, -3, -4, -2, 4, 3, 2, -3, -2, 2, -4, -2, -4, -3, 1, -2, -1, 3, 4]
  \end{align*}

  For $p=7$ we take:
  \begin{align*}
    [&\scriptstyle 2, -3, -4, -2, 4, 3, 2, -3, -2, 2, -4, -3, 1, -2, -1, 3, 4, -2, -1, -3, -4, -2, 2, -3, -2, 1, 2, -4, -2, -1, -1, -3, 1, -2, -1, 2, -3, -2, 2, -4,
      \\& \scriptstyle
       -2,-4, -3, 1, -2, -1, 3, 4, 2, -4, -3, 1, -2, -1, 3, 4, -2, 1, -2, -1, -1, 2, -4, -3, 1, -2, -1, 3, 4, -2, -3, 1, -2, -1, 2, -3, -2, 2, -4, -2,
      \\& \scriptstyle
       -4,-3, 1, -2, -1, 3, 4, 1, 2, -3, -2, -1, 1, 2, -4, -2, -1]
    \end{align*}

  For $p=8$, we take:
  \begin{align*}
    [& \scriptstyle -2, -1, -3, -4, -2, -1, -4, 2, -3, -2, 2, -4, -3, 1, -2, -1, 3, 4, -2, 1, -2, -1, -3, -4, -2, 4, 3, -2, 1, 2, -3, -2,-1, 1, 2, -4, -2, -1, -1,
      \\&
      \scriptstyle
      -3, -4, -2, 1, 2, -3, -2, -1, 1, 2, -4, -2, -1, -1, -3, -4, 2, -3, -2, 2, -4, -2, -4, -3, 1, -2, -1, 3, 4, -2, 1, 2, -3, -2, -1, 1, 2, -4, -2, 
            \\&
      \scriptstyle
      -1,-3, -4, -2, 4, 3, 2, -4, -3, 1, -2, -1, 3, 4, -2, -1, -3, 1, 2, -4, -2, -1]
  \end{align*}

  For $p=9$, we find:
  \begin{align*}
    [& \scriptstyle
      -2, -1, -3, -4, -2, 4, 3, -2, 2, -4, -3, 1, -2, -1, 3, 4, -2, -2, 2, -4, -3, 1, -2, -1, 3, 4, -2, 1, -2, -1, -1, -4, -3, 1, -2, -1, 3, 4, 2, -3,
           \\&
      \scriptstyle
      -2, 2, -4, -2, -4, -3, 1, -2, -1, 3, 4, 1, -2, -1, 2, -3, -2, 2, -4, -2, -4, -3, 1, -2, -1, 3, 4, 2, -4, -3, 1, -2, -1, 3, 4, -2, -1, -3, -4, 2, 
           \\&
      \scriptstyle
      -4,-3, 1, -2, -1, 3, 4, -2, -1, -3, -4, -3, -4, -2, 4, 3, -2, 1, 2, -3, -2, -1, 1, 2, -4, -2, -1, -3, -4, -2, 4, 3, -2, 1, 2, -3, -2, -1, 1, 2, 
      \\&
      \scriptstyle
      -4, -2,-1, 1, -2, -1, -3, -4, 2, -3, -2, 2, -4, -2, -4, -3, 1, -2, -1, 3, 4, 1, 2, -3, -2, -1, 1, 2, -4, -2, -1, 2, -3, -2, 2, -4, -2, -2, 1, 2, 
      \\&
      \scriptstyle
      -3, -2, -1, 1, 2, -4, -2, -1]
    \end{align*}

For $p=10$ we find:
\begin{align*}
  [&\scriptstyle
    2, 3, -1, -4, -2, 2, -4, -3, 1, -2, -1, 3, 4, -2, -2, 1, 2, -3, -2, -1, 1, 2, -4, -2, -1, -1, 1, -2, -1, 2, -3, -2, 2, -4, -2, -4, 
    \\&
    \scriptstyle
    -3, 1, -2, -1, 3,4, 2, -4, -3, 1, -2, -1, 3, 4, -2, 1, 2, -3, -2, -1, 1, 2, -4, -2, -1, -1, -3, -4, -2, 2, -4, -3, 1, -2, -1, 3, 4, 
    \\&
    \scriptstyle
   -2,  -1, -3,1, 2, -4, -2, -1, -1, -3, 1, -2, -1, -1, -3, -4, 2, -3, -2, 2, -4, -2, -4, -3, 1, -2, -1, 3, 4, -1, 1, -2, -1, -4, 2, 
    \\&
    \scriptstyle
    -3, -2, 2, -4, -3, 1, -2,-1, 3, 4, -2]
  \end{align*}

For $p=11$, we take:
\begin{align*}  [& \scriptstyle
    2, 1, 3,-4, 1, -2, -1,-1, 1, 2,-4,-2,-1,-1,1, -2, -1, 2, -3, -2, 2, -4, -2, -4, -3, 1, -2, -1, 3, 4, -2, 2, -3, -2, 1, 2, -4, 
    \\&
    \scriptstyle
    -2, -1, -3, 1,-2,-1,-3,-4, -2, 4, 3, 2, -4, -3, 1,-2, -1, 3, 4, -2, -3, 2, -3, -2, -2, 2, -3, -2, 2, -4, -2, -2, 2, -3, -2, 1, 
    \\&
    \scriptstyle
    2, -4, -2, -1,-2,1,2,-3,-2,-1,1,2,-4,-2,-1,-1,-3,-4,2,-3,-2,-4, -3,1,  -2,-1,3,4, 1,2,-3,-2,-1,-3,-4,
   \\&
   \scriptstyle
   -2,4,3,2,-4,-2,-4,-3, 1, -2,-1,3,4,-3,1,2,-4,-2,-1,1,2,-3,-2,-1,1,2,-4,-2,-1]
  \end{align*}

For $p=12$ we find:
\begin{align*}
  [&\scriptstyle
    2, 1, 3, 4, 2, -3, -2, 2, -4, -2, 2, -4, -3, 1, -2, -1, 3, 4, -2, 1, -2, -1, -1, -3, -4, 2, -3, -2, 2, -4, -2, -4, -3, 1, -2, -1, 3, 4,
    \\& \scriptstyle
    2, -4, -3, 1,-2, -1, 3, 4, -2, 1, -2, -1, -1, -3, -4, 2, -3, -2, 2, -4, -2, -4, -3, 1, -2, -1, 3, 4, 2, -4, -3, 1, -2, -1, 3, 4, -2,
    \\& \scriptstyle
    1, -2, -1]
  \end{align*}

For $p=13$ we find:
\begin{align*}
  [&\scriptstyle
    -2, -1, -4, -3, -2, -1, -2, 1, 2, -3, -2, -1, 1, 2, -4, -2, -1, -1, -3, -2, 1, 2, -3, -2, -1, -1, -4, -3, 2, -4, -4, -2, -4, -3, 1,
    \\& \scriptstyle
     -2, -1, 3, 4, -1,-3, -4, 2, -3, -2, 2, -4, -2, -4, -3, 1, -2, -1, 3, 4, -1, 1, 2, -3, -2, -1, 1, 2, -4, -2, -1, -2, 1, 2, -3, -2, -1, 1, 2,
    \\& \scriptstyle
    -4, -2, -1, 1, -2, -1, -1, 1, 2, -4,-2, -1, 2, -3, -2, 1, 2, -4, -2, -1, 1, -2, -1, -4, 2, -3, -2, -4, 2, -3, -2, 2, -4, -3, 1, -2, -1,
    \\& \scriptstyle
    3, 4, -2]
\end{align*}

For $p=14$ we find:
\begin{align*}
  [&\scriptstyle
    -2, -1, -1, -3, -4, -4, 2, -3, -2, -4, 2, -3, -2, 2, -4, -3, 1, -2, -1, 3, 4, -2, -1, -3, -3, -4, -2, 4, 3, 2, -4, -3, 1, -2, -1, 3, 4, 
    \\ & \scriptstyle
     -2,2, -4, -2,1, 2, -3, -2, -1, -1, -3, -4, -4, 2, -3, -2, -1, 2, -3, -2, 2, -4, -2, -1, -2, 1, 2, -3, -2, -1, 1, 2, -4, -2, -1, 1,
    \\ & \scriptstyle
      -2,-1, -1, -3, -4, 2, -3, -2, 2, -4, -2, -4,-3, 1, -2, -1, 3, 4, -4, -2, 1, 2, -3, -2, -1, -2, 1, 2, -3, -2, -1, 1, 2, -4, -2, -1,
    \\& \scriptstyle
       1, -2, -1, -4,2, -3, -2, 2, -3, -2, 1, 2,-4, -2, -1, -2, 1, 2, -3, -2, -1,1, 2, -4, -2, -1, -3, -4, 1, -2, -1, -1, -3, -4, 2, -4, 
    \\ & \scriptstyle
    -2, -4, -3, 1, -2, -1, 3, 4, -2, 1, 2, -4, -2, -1, -3, -3, 1, -2, -1, -3, -4, -2, 4, 3, 2, -4, -2, -3, -2,1, 2, -4, -2, -1, -3, -2, 1, 
    \\ & \scriptstyle
    2, -4, -2, -1, -3, -2, 1, 2, -4, -2, -1, -3, -2, 1, 2, -4, -2, -1, -4, 2, -3, -2, -1, -2, 1, 2, -3, -2, -1, 1, 2, -4, -2, -1, -1, -4,
    \\& \scriptstyle
    2, -4, -2, -1, -3, -4, 2, -3, -2, 2, -4, -2, -1, -2, 1, 2, -3, -2, -1, 1, 2, -4, -2, -1, -1, -3, -4, 2, -4, -3, 1, -2, -1, 3, 4, -2,
    \\ & \scriptstyle
     -4, 2, -4, -3, 1, -2, -1, 3,4, -2, 1, -2, -1, 2, -3, -2, 2, -4, -2, -4, -3, 1, -2, -1, 3, 4, 1, -2, -1, -4, -3, 1, -2, -1, 3, 4, 1,2,
    \\ & \scriptstyle
     -3, -2, -1, -3, -4, -2, 4, 3, 2, -3, -2, -4, -3, 1, -2, -1,3, 4, -1, -3, -4, 2, -3, -2, 2, -4, -2, -4, -3, 1, -2, -1, 3, 4, -4, 1,-2,
    \\ &  \scriptstyle
     -1, -1, -3, -4, 2, -4, -2, -4, -3, 1, -2, -1, 3, 4, -1, -3, 2, -4, -3, 1, -2, -1, 3, 4, -2]
\end{align*}

For $p=15$ we find:
\begin{align*}
  [&\scriptstyle
    2, 3, 4, -2, 1, 2, -3, -2, -1, 1, 2, -4, -2, -1, -3, 1, -2, -1, -3, 1, -2, -1, -3, -4, -2, 4, 3, 1, 2, -4, -2, -1, -3, -4, -2, 4, 3, -1, -2, 1, 2,
    \\ & \scriptstyle
     -3, -2,-1, 1, 2, -4, -2, -1, 1, -2, -1, -1, -3, -4, -3, -4, -2, 4, 3, -4, -3, 1, -2, -1, 3, 4, -4, -3, 1, -2, -1, 3, 4, -1, 2, -3, -2, 2, -4,
    \\ & \scriptstyle
     -2, -1, -1, -3, -4, 2, -4, -3, 1,-2, -1, 3, 4, -2, -1, -3, -4, 2, -4, -3, 1, -2, -1, 3, 4, -2, 1, 2, -3, -2, -1, 1, 2, -4, -2, -1, -4, -2, -1,
    \\ & \scriptstyle
     -1, -3, -4, 2, -4, -2, -4, -3, 1, -2, -1, 3, 4, -1, -3, 2,-4, -3, 1, -2, -1, 3, 4, -2]
\end{align*}

For $p=16$ we find:
\begin{align*}
  [&\scriptstyle
    -2, -1, -3, -4, -2, 4, 3, -2, 2, -4, -3, 1, -2, -1, 3, 4, -2, -2, 2, -4, -3, 1, -2, -1, 3, 4, -2, 1, -2, -1, 2, -3, -2, 2, -4, -2, -4, -3, 1,
    \\& \scriptstyle
     -2,-1,3, 4, 1, 2, -3, -2, -1, 1, 2, -4, -2, -1, -1, -4, -3, 1, -2, -1, 3, 4, 2, -4, -3, 1, -2, -1, 3, 4, -2, 1, -2, -1, -1, -3, -4, 2, -3,
    \\& \scriptstyle
     -2, 2, -4, -2, -4, -3,1, -2, -1, 3, 4, -2, 1, 2, -3, -2, -1, 1, 2, -4, -2, -1, -1, -3, -4, -1, -3, -4, 2, -3, -2, 2, -4, -2, -4, -3, 1, -2,
    \\& \scriptstyle
     -1, 3, 4, -2, 1, 2, -3, -2, -1, 1, 2, -4, -2, -1, -1, -3, -4, -3, -4, -2, 4, 3, 2, -4, -3, 1, -2, -1, 3, 4, -2, 1, -2, -1, -1, -4, -3, 1, 
    \\& \scriptstyle
    -2, -1, 3, 4, -1, -4, -3, 1, -2, -1, 3, 4, 2, -3, -2, 2,-4, -2, -4, -3, 1, -2, -1, 3, 4, 1, 2, -3, -2, -1, 1, 2, -4, -2, -1, 2, -3, -2, 2,
    \\& \scriptstyle
    -4, -2, -2, 1, 2, -3, -2, -1, 1, 2, -4, -2, -1]
\end{align*}

\end{proof}

These results can be recovered by running the code available on~\cite{BQ_git}.

\bibliographystyle{amsalpha}
\providecommand{\bysame}{\leavevmode\hbox to3em{\hrulefill}\thinspace}
\providecommand{\MR}{\relax\ifhmode\unskip\space\fi MR }
\providecommand{\MRhref}[2]{\href{http://www.ams.org/mathscinet-getitem?mr=#1}{#2}
}
\providecommand{\href}[2]{#2}

\end{document}